  \documentclass[11pt,letterpaper]{amsart}

  \usepackage[margin=3cm]{geometry}
  \usepackage{amsmath,amssymb,amsthm,amsfonts,bbding,stmaryrd,dsfont,wasysym,pifont,xspace,xcolor}
  \usepackage{parskip,fancyhdr,url}
  \usepackage{comment}
  \usepackage{fourier}
  \usepackage[
     breaklinks,colorlinks,
     citecolor=blue,linkcolor=red,urlcolor=teal
  ]{hyperref}

  \usepackage{tikz}
  \usetikzlibrary{arrows,calc,decorations.pathmorphing,backgrounds,positioning,fit}

  \usepackage{etoolbox}

  \usepackage{epstopdf,yfonts}
  \usepackage{graphicx,enumerate,epsfig} 
  \usepackage{mathrsfs}

  \newtheorem{theorem}{Theorem}[section]
  \newtheorem{proposition}[theorem]{Proposition}
  \newtheorem{corollary}[theorem]{Corollary}
  \newtheorem{lemma}[theorem]{Lemma}

  \newtheorem{introthm}{Theorem}

  \newtheorem{introcor}[introthm]{Corollary}

  \theoremstyle{definition}
  \newtheorem{definition}[theorem]{Definition}
  
  \newtheorem*{claim*}{Claim}
  
  \newtheorem{example}[theorem]{Example}
  
  \newtheorem*{question*}{Question}
  \newtheorem*{answer*}{Answer}
  \newtheorem*{application*}{Application}

  \theoremstyle{remark}
  \newtheorem{remark}[theorem]{Remark}
  \newtheorem*{remark*}{Remark}

  \newcommand{\s}{\ensuremath{S}\xspace} 
  \newcommand{\T}{\ensuremath{\mathcal T}(\s)\xspace} 
  \newcommand{\TC}{\ensuremath{\overline{\T}}\xspace} 
   
  \newcommand{\mcg}{\ensuremath{\mathcal{MCG}(\s)}\xspace}  
  \newcommand{\cc}{\ensuremath{\mathcal C}(\s)\xspace} 
  \newcommand{\curve}{\ensuremath{\mathcal S}(S)\xspace} 
   \newcommand{\MF}{\ensuremath{\mathcal{M}\!\!\mathcal{F}}\xspace} 
  \newcommand{\ML}{\ensuremath{\mathcal{M}\!\!\mathcal{L}}(S)\xspace} 
  \newcommand{\PMF}{\ensuremath{\mathcal{P}\!\!\!\! \mathcal{M}\!\! \mathcal{F}}\xspace}  
  \newcommand{\PML}{\ensuremath{\mathcal{P}\!\!\!\! \mathcal{M}\!\! \mathcal{L}}(S)\xspace}  
  
  \newcommand{\Teich}{{Teichm\"uller }} 
  \newcommand{\cat}{\ensuremath{\text{CAT}(0)}\xspace}
  
  \newcommand{\RR}{\ensuremath{\mathbb{R}}\xspace}
  \newcommand{\ZZ}{\ensuremath{\mathbb{Z}}\xspace}
  \newcommand{\HH}{\ensuremath{\mathbb{H}}\xspace}

  \DeclareMathOperator{\I}{i}

  \DeclareMathOperator{\Ext}{Ext}
  \DeclareMathOperator{\isom}{Isom}

  \newcommand{\param}{{\mathchoice{\mkern1mu\mbox{\raise2.2pt\hbox{$
  \centerdot$}}
  \mkern1mu}{\mkern1mu\mbox{\raise2.2pt\hbox{$\centerdot$}}\mkern1mu}{
  \mkern1.5mu\centerdot\mkern1.5mu}{\mkern1.5mu\centerdot\mkern1.5mu}}}

  \newcommand{\set}[2]{\ensuremath{\big\{ \,\, {#1} \,\,:\,\, {#2}
  \,\, \big\}}\xspace}
   
  \newcommand{\dL}{\ensuremath{d_{\text{Th}}}\xspace}
  \newcommand{\dT}{\ensuremath{d_T}\xspace}

  \fancypagestyle{plain}

\begin{document}


  \title    {$\star$ Stars at infinity in the Thurston boundary $\star$} 
  \author   {Meenakshy Jyothis}
  \address{University of Oklahoma \\
    Department of Mathematics\\
    601 Elm Avenue Room 423\\Norman, OK 73019}
\email{mjyothis@ou.edu}
  \author   {Jing Tao}
  \address{University of Oklahoma \\
    Department of Mathematics\\
    601 Elm Avenue Room 423\\Norman, OK 73019}
\email{jing@ou.edu}
 

  \begin{abstract} 
     
    We study the stars at infinity in the Thurston boundary of \Teich space for the Thurston and \Teich metrics. For the Thurston metric, we prove that for every finite-type surface, the star of a projective measured lamination is exactly its zero set, extending a theorem of Liu-Shi from closed surfaces. Moreover, the based star agrees with the star. For the \Teich metric, we show that both the based star and the star of a projective measured lamination coincide with its \emph{two-step} zero set. This set can be strictly larger than the zero set, disproving a conjecture of Karlsson.

  \end{abstract}
  
  \maketitle
  

\section{Introduction}

Let $S=S_{g,p}$ be a connected orientable surface of finite type with complexity $3g-3+p \ge 1$, and let \T denote its \Teich space. Thurston \cite{FLP79, Thu88} constructed a mapping-class-group $\mcg$--equivariant compactification $\TC = \T \sqcup \PML$, whose boundary is the space \PML of projective measured laminations. This compactification plays a central role in the geometry and dynamics of \mcg.

In \cite{Kar05}, Karlsson introduced the notion of \emph{stars at infinity} in a compactified metric space. Informally, given a boundary point $\xi$, its star $\star(\xi)$ consists of the boundary points that cannot be separated from $\xi$ by asymptotic half-spaces. There is also a based star $\star_b(\xi)$, which a priori is only a subset of $\star(\xi)$ and may depend on the choice of basepoint $b$. Half-spaces and stars are useful tools for studying dynamical properties of isometries. A boundary point $\xi$ is called \emph{hyperbolic} if $\star(\xi) = \{\xi \}$, and an isometry is \emph{strictly hyperbolic} if its limit set in the boundary consists of two distinct hyperbolic points. This notion encapsulates isometries that act with ``north-south'' dynamics; see for instance \cite[Proposition 9]{Kar05}.

In the same paper, Karlsson studied stars in the Thurston boundary when \T is equipped with the \Teich metric. For a non-zero element $\xi \in \ML$, denote by $[\xi]$ its projective class, and define the zero set of $[\xi] \in \PML$ by \[ Z \big( [\xi] \big) = \big\{ [\eta] \in \PML: i (\eta,\xi)=0 \big\},\]
where $i(\param,\param)$ denotes the geometric intersection pairing on (unprojectivized) measured laminations. Karlsson conjectured that $\star^T \big( [\xi] \big) = Z \big( [\xi] \big)$, for every $[\xi] \in \PML$. He proved this when $\xi$ is uniquely ergodic. Duchin and Fisher \cite{DF21} subsequently proved the inclusion $Z \big( [\xi] \big) \subset \star^T \big( [\xi] \big)$. In particular, when $S$ has complexity at least two, these results identify the uniquely ergodic projective measured laminations as the hyperbolic points for the \Teich metric.

Our first main theorem shows that the reverse inclusion of Karlsson’s conjecture fails in general. We say $\xi, \eta \in \ML$ \emph{do not jointly fill} $S$ if \[ \exists \lambda \in \ML-\{0\} \quad \text{such that} \quad i(\xi,\lambda)=i(\eta,\lambda)=0.\] Define the \emph{two-step zero set} of $[\xi] \in \PML$ by
\[ Z^{(2)}\big( [\xi] \big) = \Big \{ [\eta] \in \PML: \text{ $\xi$ and $\eta$ do not jointly fill $S$} \Big\}. \]

  \begin{introthm} \label{introthm:false}
    Let $\T$ be equipped with the \Teich metric. For  all $b \in \T$ and all $[\xi] \in \PML$, \[ \star_b^T \big( [\xi] \big) = \star^T \big([\xi] \big) = Z^{(2)}\big( [\xi] \big).\] 
  \end{introthm}

  Thus the \Teich star can be strictly larger than the zero set.  For example, if $\alpha$ and $\beta$ are intersecting simple closed curves but are disjoint from a common curve $\gamma$, then $[\beta] \in \star_b^T \big( [\alpha]\big)$. We briefly sketch the proof in this case. The goal is to construct sequences in $\T$ by pinching $\gamma$ at a superexponential rate while applying exponentially growing powers of Dehn twists about $\alpha$ and $\beta$. The twisting determines the limiting projective lamination, whereas the pinching dominates the \Teich distance from $b$. This produces sequences $x_n, z_n \in \T$ converging to $[\alpha]$ and $[\beta]$, respectively, such that $x_n$ and $z_n$ eventually lie in the same half-space. This construction extends to multicurves and then by approximation to all laminations, yielding the direction $Z^{(2)}\big( [\xi] \big) \subset \star_b^T \big([\xi] \big)$. For the reverse inclusion, Kerckhoff's extremal-length formula shows that jointly filling laminations can be separated by asymptotic half-spaces.

  Another principal subject of this paper is the corresponding problem for the \emph{Thurston metric} on \T.  Walsh showed that the horofunction compactification of $\T$ equipped with the Thurston metric is naturally identified with the Thurston compactification \cite{Wal14}, making \PML the natural metric boundary for the Thurston metric.
 
  The relation between stars and zero sets for the Thurston metric was previously studied by Liu and Shi \cite{LS24}. Using Bonahon’s theory of geodesic currents, they proved that, for a closed surface of genus at least two, $\star^{\text{Th}} \big( [\xi] \big) = Z \big( [\xi] \big)$ for all $[\xi] \in \PML$, thus verifying the Thurston metric-analogue of Karlsson's conjecture. 
   
  Their proof uses geodesic currents. Because continuity of intersection pairing on currents plays a crucial role, their proof a priori only works for closed surfaces; indeed Liu and Shi left the analogous statement for punctured surfaces open.
  
  In this paper we give a different proof of Liu and Shi's main theorem, valid for all finite-type surfaces. We also prove the stronger equality with the based star $\star_b^{\text{Th}}\big( [\xi \big)$.

 \begin{introthm} \label{introthm:stars}
  Let $\T$ be equipped with the Thurston metric. For  all $b \in \T$ and all $[\xi] \in \PML$, \[ \star^{\text{Th}}_b \big( [\xi] \big) = \star^{\text{Th}} \big( [\xi] \big) = Z \big([\xi] \big).\]
\end{introthm}

  There are two directions, requiring different methods. To prove the inclusion $\star^{\text{Th}} \big( [\xi] \big) \subset Z \big( [\xi] \big)$, we use short markings and Walsh’s normalized length functions. For the inclusion $Z \big( [\xi] \big) \subset \star_b^{\text{Th}} \big([\xi] \big)$, we use Bonahon's shearing coordinates \cite{Bon96} and Papadopoulos's characterization of limiting behavior of paths in shearing coordinates \cite{Pap91}. 

  Theorem \ref{introthm:stars} gives an immediate characterization of the hyperbolic boundary points.

  \begin{introcor} \label{introcor:hyperbolic_points}

    Let $\T$ be equipped with the Thurston metric. If $S$ is the once-punctured torus or four-punctured sphere, then every $[\xi] \in \PML$ is hyperbolic. If $S$ has complexity at least $2$, then a point $[\xi] \in \PML $ is hyperbolic if and only if $\xi$ is a uniquely ergodic lamination.

  \end{introcor}

    Using the Thurston compactification of $\T$, Thurston classified elements of \mcg as either finite order, reducible or pseudo-Anosov. The action of $\mcg$ on \T is by isometries of the Thurston metric. We have:

  \begin{introcor} \label{introcor:PA}
     
     An element $f \in \mcg$ is a strictly hyperbolic isometry of the Thurston metric on $\T$ if and only if $f$ is
     pseudo-Anosov.

  \end{introcor}
  
  Corollary \ref{introcor:PA} should be contrasted with the fact that there are reducible elements of \mcg that act on \T as \emph{hyperbolic isometries} in the
  Thurston metric, in the sense that they have positive translation length and geodesic axes \cite{LPST13, HT}.
  
  Theorem \ref{introthm:false} and \ref{introthm:stars} also identify the incidence relation induced by stars on the set of simple closed curves. Suppose $S$ has complexity at least $2$, Let $\mathcal{C}(S)$ be the curve graph of $S$. Then for any distinct curves $\alpha$ and $\beta$, \[ [\beta] \in \star^{\text{T}} \big( [\alpha] \big) \Longleftrightarrow d_{\mathcal{C}(S)}(\alpha,\beta) \le 2 \quad \text{and} \quad [\beta] \in \star^{\text{Th}} \big( [\alpha] \big) \Longleftrightarrow d_{\mathcal{C}(S)}(\alpha,\beta) = 1.\] Thus, for the Thurston metric, the graph induced by the star relation on curve vertices is precisely $\mathcal{C}(S)$. This allows us to recover the isometry group of the Thurston metric.

  \begin{introcor} \label{introcor:isometries}
    Suppose that $S$ has complexity $3g-3+p \ge 2$ and is not the twice-punctured torus. Every isometry of the Thurston metric on \T is induced by an element of the extended mapping class group of $S$.
  \end{introcor}

  Corollary \ref{introcor:isometries} was first proved by Walsh \cite{Wal14} using the horofunction compactification of \T. The conclusion of the corollary also holds for the once- and twice-punctured tori, as well as the four-punctured sphere; see \cite{DLRT20,Pan23}. These exceptional surfaces are excluded from our argument because it relies on the fact that every automorphism of $\mathcal{C}(S)$ is induced by an element of the extended mapping class group.
  
  The paper is organized as follows. Section \ref{sec:background} reviews horofunction compactification, measured laminations, the Thurston compactification, and shearing coordinates. Section \ref{sec:stars} recalls stars at infinity and establishes some useful criteria. In Section \ref{sec:ThurstonMetric}, we review the Thurston metric and prove Theorem \ref{introthm:stars}. In Section \ref{sec:Teich}, we review the \Teich metric and prove Theorem \ref{introthm:false}.

  \subsection*{Acknowledgements and AI Use} 
  
  The authors acknowledge ChatGPT for suggesting relevant references and assisting in improving the exposition and presentation of some proofs. The authors take full responsibility for the mathematical content and the final text. The authors gratefully acknowledge support from an AMS-Simons Travel Grant and NSF grant DMS-2304920.
  
\section{Preliminaries} 

  \label{sec:background}
  
  We recall the background needed below; see \cite{FLP79, CB88, Thu88, Bon88, Bus92} for further details on hyperbolic geometry, \Teich theory, and measured laminations and foliations. 

  \subsection{Horofunction compactification.} 
  
  The notion of horofunction compactification of a metric space was first introduced by Gromov in \cite{Gro81}; see also \cite{BGS85,Bal95}. 

  Let $(M,d)$ be a metric space, and let $\bar d(x,y)=d(x,y)+d(y,x)$ be the \emph{symmetrization}. Throughout this paper, we use $\bar d$ to define the topology, completeness, and properness of $M$. We say $M$ satisfies the Busemann axiom if for all $x$ and all sequences $x_n$ in $M$, $d(x_n,x) \to 0$ if and only if $d(x,x_n) \to 0$. 



  Let $C \big(M, \RR \big)$ be the space of continuous real-valued  functions on $M$ equipped with the topology of uniform convergence on compact sets. Fix a basepoint $b \in M$. For any $z \in M$, define $\Phi_z : M \to \RR$ to be \[ \Phi_z(x) = d(x,z) - d(b,z).\] This defines a map $\Phi: M \to C \big(M, \RR \big)$ by $z \mapsto \Phi_z$. 

  \begin{lemma}
    Let $(M,d)$ be a metric space and $b \in M$. The following statements hold:
    \begin{itemize}
    \item For all $z \in M$, $\Phi_z$ is $1$--Lipschitz with respect to $\bar d$; namely, for all $x,y \in M$, \[ |\Phi_z(x) - \Phi_z(y)| \le \max \{ d(x,y), d(y,x) \} \le \bar{d}(x,y). \] 
    \item The map $\Phi: M \to C(M,\RR)$ is injective and continuous, and the closure $\overline{\Phi(M)}$ of $\Phi(M)$ in $C(M,\RR)$ is compact, and its topology is independent of the basepoint $b \in M$.
    \end{itemize}
  \end{lemma}

   When $\Phi: M \to C(M,\RR)$ is a topological embedding, we will identify $M$ with its image in $C(M,\RR)$, and call its closure $\overline{M}$ in $C(M,\RR)$ the  \emph{horofunction compactification} of $M$, and $\partial M = \overline{M} \setminus M$ the \emph{horofunction boundary} of $M$. An element of $\partial M$ is called a \emph{horofunction}. 
   
   \begin{proposition} \label{prop:assumptions}
   If $(M,d)$ is a proper geodesic metric space satisfying the Busemann axiom, then $\Phi: M \to C(M,\RR)$ is a topological embedding, $\overline{M}$ and $\partial M$ are compact metrizable, and any isometry of $M$ extends continuously to a homeomorphism of $\overline{M}$.
   \end{proposition}

  \begin{remark}
    For asymmetric metrics, the order of the variables in the definition of $\Phi_z$ matters. We use the convention above throughout.
  \end{remark}

  \subsection*{Curves and surfaces}
  Let $S=S_{g,p}$ be a connected oriented finite-type surface with complexity $3g-3+p\ge 1$. The \Teich space \T is the space of marked hyperbolic surfaces homeomorphic to $S$. The mapping class group \mcg, the group of isotopy classes of orientation-preserving homeomorphisms of $S$, acts on \T by change of marking. Let $\curve$ denote the set of isotopy classes of essential simple closed curves. For $x\in \T$ and $\alpha, \beta \in \curve$, let $\ell_x(\alpha)$ be the hyperbolic length of the geodesic representative of $\alpha$, and let $i(\alpha,\beta)$ denote geometric intersection number. A multicurve is a collection of pairwise disjoint essential, simple closed curves. A finite collection $\gamma=\{\gamma_1,\ldots,\gamma_k\}$  \emph{fills} $S$ if $\sum_i i(\gamma_i,\alpha)>0$ for every $\alpha\in\curve$.
  
  \begin{lemma}[Collar Lemma] \label{lem:collar}
    Let $\alpha$ be a simple closed geodesic on $x \in \T$. Denote by $U_x(\alpha)$ the standard collar about $\alpha$, namely the collar of radius \[ r_x(\alpha) = \sinh^{-1}\left(\frac{1}{\sinh(\ell_x(\alpha)/2)}\right). \] The standard collars about disjoint curves are disjoint. Moreover, there exist $C_0,\epsilon_0>0$ such that if $\ell_x(\alpha)\le\epsilon_0$, then each boundary component of $U_x(\alpha)$ has length at most $C_0$. 
  \end{lemma} 
  
  Let $r=r_x(\alpha)$ be the radius of the collar $U_x(\alpha)$. Choose an orientation of $\alpha$ and identify the universal cover of $U_x(\alpha)$ with $[-r,r]\times\mathbb R$, with deck transformation acting by $ (s,t)\longmapsto (s,t+\ell_x(\alpha))$. Given a simple geodesic arc $\sigma$ joining the two boundary components of $U_x(\alpha)$, a lift of $\sigma$ has endpoints $(-r,t_-)$ and $(r,t_+)$. We define the \emph{winding number} of $\sigma$ about $\alpha$ to be \[ w(\sigma,\alpha)=\frac{t_+-t_-}{\ell_x(\alpha)}. \] A Dehn twist about $\alpha$ changes this number by one, up to sign. 
  
  \begin{lemma} \label{lem:winding}
    For every $L>0$ there is $C=C(L)$ such that, whenever $\ell_x(\alpha)\le L$ and $\sigma$ is a simple arc that joins the two boundary components of the standard collar $U_x(\alpha)$ of $\alpha$, then \[ \Big| \ell_x(\sigma) -\big(2r_x(\alpha) +|w(\sigma,\alpha)|\ell_x(\alpha)\big) \Big| \le C. \]
  \end{lemma}

  \subsection*{Measured laminations}
  
  Let $\ML$ and $\PML$ denote the spaces of measured laminations and projective measured laminations. We use $|\mu|$ for the support of $\mu$. A lamination is \emph{minimal} if every half-leaf is dense in its support, \emph{filling} if each complementary component is a disk or once-punctured disk, and \emph{maximal} if every complementary region is an ideal triangle. When $S$ is punctured, a maximal lamination with no closed leaves whose leaves limit to punctures is called \emph{maximal ideal}. A measured lamination is \emph{uniquely ergodic} if its support is minimal and filling and supports a unique transverse measure up to scale. Weighted curves are dense in $\ML$, and geometric intersection extends continuously and homogeneously to 
  \[ i:\ML\times\ML\longrightarrow\mathbb R_{\ge0}. \] 
  Likewise, hyperbolic length extends continuously to 
  \[ \ell_x:\ML\longrightarrow\mathbb R_{\ge0}. \] 
  Denote by $\ML^*  = \ML \setminus \{0\}.$ When the complexity of $S$ is at least $2$, $\mu \in \ML^*$ is uniquely ergodic if and only if \[ i(\mu,\lambda)=0 \quad\Longrightarrow\quad \lambda=c\mu \quad \text{ for some } c\ge 0.\]
  
  \begin{lemma} \label{lem:non-filling}
    Let $\lambda\in\ML^*$ be non-filling. Then there is a curve $\gamma$ with $i(\lambda,\gamma)=0$ such that \[ i(\mu,\gamma)>0 \quad\Longrightarrow\quad i(\mu,\lambda)>0 \] for every $\mu\in\ML$. 
  \end{lemma} 
  
  \begin{proof} 
    
    Choose a minimal component $\lambda_0$ of $\lambda$. If $\lambda_0$ is a curve, take $\gamma=\lambda_0$. Otherwise $\lambda_0$ fills a proper subsurface $Y\subset S$; taking $\gamma$ to be a component of $\partial Y$ gives the conclusion. \qedhere
    
  \end{proof} 

  \begin{theorem}[{\cite[Theorem C]{LM10}}] \label{thm:simultaneous}
   
   Suppose $[\xi], [\eta] \in \PML$ have the same support. Then there are weighted multicurves $A_n$ and $B_n$ with the same support such that $[A_n] \to [\xi]$ and $[B_n] \to [\eta]$.

  \end{theorem}

  \subsection*{The Thurston compactification} 
  
  Thurston gave a compactification of $\T$ by $\PML$ such that \[ \TC = \T \sqcup \PML\] is homeomorphic to a closed ball of dimension $6g-6+2p$, which is compatible with the $\mcg$ action on $\T$ and $\PML$. The topology of the compactification is as follows. A sequence $x_n\in \T$ converges to $[\mu]\in\PML$ if and only if there are $r_n>0$ such that \[ r_n\ell_{x_n}(\alpha)\longrightarrow i(\mu,\alpha) \] for every $\alpha\in \curve$. Necessarily $r_n\to 0$. 
  
  \subsection*{Shearing coordinates} 
  
   Let $\MF(S)$ denote the space of (singular) measured foliations on $S$. There is a natural homeomorphism \[ \MF(S) \to \ML, \quad F \mapsto \lambda_F \] obtained by "straightening" the leaves of $|F|$ \cite{Lev83}. This homeomorphism preserves intersection number and descends to a homeomorphism from $\PMF(S) \to \PML$.  We will simultaneously view $\PMF(S)$ as the Thurston boundary of $\T$.

   Given a maximal geodesic lamination $\lambda$, let $\MF(\lambda)$ be the set of \emph{standard} measured foliations transverse to $\lambda$, up to isotopy relative to $\lambda$ and Whitehead moves. Standard means leaves near the punctures of $S$ are closed. Given $x \in \mathcal{T}(S)$, realize $\lambda$ geodesically on $x$. Since $\lambda$ is maximal, the complementary components of $\lambda$ are ideal triangles. Each complementary ideal triangle carries a canonical horocyclic measured foliation. These local foliations fit together across the leaves of $\lambda$ to produce a standard measured foliation $F_\lambda(x)$ transverse to $\lambda$, whose transverse data record the relative shears between complementary triangles. The resulting map
  \[\operatorname{Sh}_{\lambda}\colon \T\to \MF(\lambda),
   \qquad
   x\mapsto F_\lambda(x),
   \]
is called the $\lambda$-shearing-coordinate map.

  \begin{theorem}[\cite{Thu86, Bon96}]
  For every maximal geodesic lamination $\lambda$, the $\lambda$-shearing-coordinate map
  $
   \operatorname{Sh}_{\lambda}\colon
   \T\to \MF(\lambda)$
  is a homeomorphism.
   \end{theorem}

  Following Thurston \cite{Thu86}, a measured lamination $\eta$ is \emph{totally transverse} to $\lambda$ if every leaf of $|\eta|$, viewed as a parametrized geodesic $\mathbb{R} \to S$, meets $\lambda$ transversely infinitely often in both directions, and every half-leaf of $\lambda$ not tending to a puncture meets $|\eta|$ infinitely often. 
  
  \begin{proposition} \label{prop:transversality}\cite[Theorem 9.4]{Thu86} 
  Let $\lambda$ be maximal. A measured foliation $F$ has a standard representative transverse to $\lambda$ if and only if its straightening $\lambda_F$ is totally transverse to $\lambda$. 
  
  \end{proposition} 
  
  If $\lambda$ has minimal filling support, then every $\eta\in\ML$ with $|\eta|\neq|\lambda|$, equivalently $i(\eta,\lambda)>0$, is totally transverse to $\lambda$. Hence: 
  
  \begin{corollary} \label{cor:unique_erg}
  
    If $S$ is closed and $\lambda$ is a maximal uniquely ergodic lamination, then \[ \MF(\lambda)\cong \ML\setminus\mathbb R_{\ge0}\lambda. \] 
  \end{corollary} 
  
  \begin{theorem}[\cite{Pap88,Pap91}] \label{thm:papa}
  
    Let $\lambda$ be a measurable maximal lamination when $S$ is closed, or a maximal ideal lamination when $S$ is punctured. Suppose $x_n\in\T$ has $\lambda$-shearing coordinates $F_n\in\MF(\lambda)$ and $F_n$ leaves every compact subset of $\MF(S)$. Then for every curve $\alpha$ there is $C_\alpha$ such that \[ i(F_n,\alpha) \le \ell_{x_n}(\alpha) \le i(F_n,\alpha)+C_\alpha. \] Consequently, if $[F_n]\to[F]\in\PML$, then $x_n\to[F]$ in the Thurston compactification. 
  
  \end{theorem}

\section{Stars at infinity}
  
  \label{sec:stars}

  We recall the definition of stars at infinity from \cite{Kar05}. Let $(M,d)$ be a complete metric space, possibly asymmetric, with basepoint $b\in M$, and let $\overline M$ be a bordification of $M$ with boundary $\partial M=\overline M\setminus M$. We define $d(x,\xi)=d(\xi,x) = \infty$ for any $x \in M$ and $\xi \in \partial M$. This is consistent with the completeness of $M$.

  \begin{definition}
     
     For $V \subset \overline{M}$ and $C \ge 0$, the
     $C$-\emph{half-space} defined by $V$ relative to $b$ is \[ H_b(V,C) =
     \set{z\in M}{d(V,z) \le d(b,z) + C},\] 
     where \[ d(V,z) = \inf_{x \in V} d(x,z).\] 
  \end{definition}
  
  \begin{definition}
    
     Let $\xi \in \partial M$ and let $V(\xi)$ be a collection of neighborhoods of $\xi$. 
     Given $C \ge 0$, define \[ \star(\xi,C) = \bigcap_{V \in V(\xi)} \overline{H_b(V,C}).\]  We also set $\star_b(\xi) = \star(\xi,0)$, called the \emph{star of}  $\xi$ \emph{based at} $b$.
     The \emph{star} of $\xi$ is defined to be 
     \[ 
        \star(\xi) = \overline{\bigcup_{C \ge 0} \bigcap_{V \in V(\xi)}
        \overline{ H_b(V,C) }}.
     \] 

  \end{definition}
  
  The set $\star(\xi)$ is independent of the basepoint $b$ \cite[Lemma 6]{Kar05} and one can choose $V(\xi)$ to be a fundamental family of  neighborhoods of $\xi$. 

  We will repeatedly use the following sequential characterization.
  
  \begin{lemma} \label{lem:necc_star}
  
    Let $\xi,\eta\in\partial M$. If there are sequences $x_n,z_n\in M$ with $x_n\to\xi$, $z_n\to\eta$, and \[ d(x_n,z_n)-d(b,z_n)\le C \] for all sufficiently large $n$, then $\eta\in\star(\xi,C)$. Conversely, if $\overline M$ is first countable and $\eta\in\star(\xi,C)$, then there are sequences $x_n\to\xi$ and $z_n\to\eta$ such that \[ \limsup_{n\to\infty} \bigl(d(x_n,z_n)-d(b,z_n)\bigr) \le C. \] 
    
  \end{lemma}

  \begin{proof} 
    For the first direction, if $V$ is a neighborhood of $\xi$, then $x_n\in V$ eventually, so \[ d(V,z_n)-d(b,z_n) \le d(x_n,z_n)-d(b,z_n) \le C. \] Since $z_n\to\eta$, this gives $\eta\in\overline{H_b(V,C)}$, and hence $\eta\in\star(\xi,C)$.
    
    Conversely, choose nested neighborhood bases $V_n$ and $U_n$ at $\xi$ and $\eta$. Since $\eta\in\overline{H_b(V_n,C)}$, choose $z_n\in U_n\cap H_b(V_n,C)$ and then $x_n\in V_n$ such that \[ d(x_n,z_n)-d(b,z_n)\le C+\frac1n. \] Then $x_n\to\xi$, $z_n\to\eta$, and the conclusion follows. \qedhere
    
  \end{proof}

  We have the following based-star variant of semicontinuity; cf. \cite{DF21}.
 
  \begin{lemma}[Semicontinuity for based stars] \label{lem:semicontinuity}
    Let $(M,d)$ be a metric space with a metrizable bordification
    $\overline M$, and fix $b\in M$. Suppose
    \[
    \xi_n\to\xi,\qquad \eta_n\to\eta, \quad \text{where}\quad \xi_n, \eta_n, \xi, \eta \in \partial M.
    \]
    For every $n$, suppose there are sequences
    $x_{n,k}\to\xi_n$ and $z_{n,k}\to\eta_n$ such that
    \[
    \liminf_{k\to\infty}
    \bigl(d(x_{n,k},z_{n,k})-d(b,z_{n,k})\bigr)<0.
    \]
    Then $ \eta\in\star_b(\xi).$
    \end{lemma}

    \begin{proof}
    Fix a metric $\rho$ inducing the topology on $\overline M$. For each
    $n$, set
    \[
    a_{n,k}
    =
    d(x_{n,k},z_{n,k})-d(b,z_{n,k}).
    \]
    By assumption, $\liminf_{k\to\infty} a_{n,k}<0.$
    Hence there are arbitrarily large $k$ for which $a_{n,k}<0$.
    
    Since
    \[
    x_{n,k}\to\xi_n
    \qquad\text{and}\qquad
    z_{n,k}\to\eta_n
    \]
    as $k\to\infty$, we may choose $k(n)$ sufficiently large so that
    simultaneously
    \[
    \rho(x_{n,k(n)},\xi_n)<\frac1n,
    \qquad
    \rho(z_{n,k(n)},\eta_n)<\frac1n,
    \]
    and
    \[
    d(x_{n,k(n)},z_{n,k(n)})
    -
    d(b,z_{n,k(n)})<0.
    \]
    Set
    \[
    x_n=x_{n,k(n)},
    \qquad
    z_n=z_{n,k(n)}.
    \]
    
    Since $\xi_n\to\xi$ and $\eta_n\to\eta$, the first two inequalities
    imply
    \[
    x_n\to\xi,
    \qquad
    z_n\to\eta.
    \]
    Moreover, 
    \[
    d(x_n,z_n)-d(b,z_n)<0, \,\, \forall n.
    \]
    Lemma \ref{lem:necc_star}, applied with $C=0$,
    therefore gives
    \[
    \eta\in\star(\xi,0)=\star_b(\xi). \qedhere
    \]
    \end{proof}
  
  \begin{definition}
     
     A point $\xi \in \partial M$ is called a \emph{hyperbolic point} if
     $\star(\xi) = \{\xi\}$.
     
  \end{definition}
  
  Let $\isom(M)$ be the group of isometries of $M$. Given $g \in \isom(M)$,
  the \emph{limit set} $\Lambda(g)$ of $g$ is \[ \Lambda(g) =
  \overline{\set{g^n(x)}{n \in \ZZ}} \bigcap \partial M.\]

  \begin{definition}
     
    An element $g \in \isom(M)$ is called \emph{strictly hyperbolic} if $\Lambda(g)$ consists of exactly two distinct hyperbolic points. 
     
  \end{definition}

   Using stars, \cite{Kar05} defined an extended metric on $\partial M$ as follows. Define an equivalence relation on $\partial M$, where $\xi \sim \eta$ if $\star(\xi)=\star(\eta)$. Let $\bar{\xi}$ be the class of $\xi$ in $\partial M/\sim$. Consider the graph with vertex set the elements of $\partial M/\sim$, and two vertices $\bar{\xi}$ and $\bar{\eta}$ span an edge of length $1$ if either $\eta \in \star(\xi)$ or $\xi \in \star(\eta)$. Let $d_\star$ be the associated path metric on $\partial M/\sim$, where we assign distance $\infty$ to vertices lying in distinct connected components.
  
  \begin{example}
     
     We list some examples of metric spaces whose stars at infinity are
     understood \cite{Kar05}.
  
     Let $M=\HH^n$ be hyperbolic $n$--space and let $\partial \HH^n \cong
     S^{n-1}$ be the Gromov boundary. Then for any $\xi \in \partial
     \HH^n$, $\star(\xi)=\star_b(\xi) = \{ \xi \}$. Thus, every point in $\partial
     \HH^n$ is hyperbolic, every hyperbolic isometry of $\HH^n$ is
     strictly hyperbolic, and the star metric $d_\star$ on $\partial \HH^n$ is discrete, taking value $\infty$ on any pair of distinct points. These results also hold more generally for Gromov hyperbolic spaces. 

     To contrast, let $M=\RR^n$ be Euclidean $n$--space and let $\partial
     \RR^n \cong S^{n-1}$ be the visual sphere at infinity. Here, for any $\xi
     \in \partial \RR^n$, $\star(\xi)=\star_b(\xi)$ is the hemisphere centered at $\xi$. So points of $\partial \RR^n$ are not hyperbolic, isometries of $\RR^n$ are not strictly hyperbolic, and the $d_\star$-diameter of $\partial \RR^n$ is two.

     If $M$ is a complete \cat space and $\partial M$ is the visual boundary of $M$ equipped with angular metric, then $\star(\xi) = \{ \eta \in \partial M : \angle(\eta, \xi) \le \pi/2\}$ for any $\xi \in \partial M$, and the strictly hyperbolic isometries are precisely the rank-one isometries of $M$.

  \end{example}

  \subsection{Stars and horofunction compactification}
  
  Let $M$ be a complete metric space, with possibly an asymmetric metric $d$, and $\overline{M}$ its horofunction compactification relative to a basepoint $b \in M$.
  
  If $d$ is asymmetric, then the definition of half-spaces in $M$ depends on the order of the points. Our convention is chosen to fit with our definition of $\Phi_z$ given in the previous section. In particular, an equivalent definition of a $C$-half-space is \[ H_b(V,C) = \Big\{ z \in M: \inf_{x \in V\cap M} \Phi_z(x) \le C \Big\}.\]
  
\begin{lemma} \label{lem:suff_star}
  Let $\xi,\eta \in \partial M$. If for every neighborhood $V$ of $\xi$, $\inf_{x \in V\cap M} \Phi_\eta(x) < C$, then $\eta \in \star(\xi,C)$. 
\end{lemma}

\begin{proof}
  Choose $x \in V \cap M$ such that $\Phi_\eta(x) < C $. Let $U \subset \overline{M}$ be a neighborhood of $\eta$ such that for all $z \in U$, $\Phi_z(x) < C$. Therefore, for all $z \in U \cap M$, we have \[ \inf_{x \in V\cap M} \Phi_z(x) \le \Phi_z(x) < C. \] In other words, $U \cap M \subset H_b(V,C)$, so $\eta \in \overline{H_b(V,C)}$. Since $V$ was arbitrary, $\eta \in \star(\xi,C).$
\end{proof}

\section{Thurston metric} \label{sec:ThurstonMetric}

    Thurston's asymmetric metric on $\T$ is defined by
    \[ \dL(x,y) = \log \inf_\varphi \text{L}(\varphi),
    \]
    where $\varphi:x\to y$ ranges over $L$--Lipschitz homeomorphisms in the appropriate
    homotopy class. Thurston proved the length-ratio formula \cite{Thu86}
    \begin{equation} \label{eqn:ratio_curves}
    \dL(x,y) =
    \log\sup_{\alpha\in \curve}
    \frac{\ell_y(\alpha)}{\ell_x(\alpha)}.
    \end{equation}
     Since hyperbolic length extends continuously to $\ML$ and is homogeneous, Thurston's formula can equivalently be maximized over any compact section of $\PML$ in $\ML$.

  \subsection*{Horofunction compactification of $\T$}

  Fix a basepoint $b \in \T$. By the formula given by (\ref{eqn:ratio_curves}), 
  \[ \Phi_z(x) = \dL(x,z) - \dL(b,z) = \log \sup_{\alpha \in \curve} \frac{\ell_z(\alpha)}{\ell_x(\alpha)} - \log \sup_{\alpha \in \curve} \frac{\ell_z(\alpha)}{\ell_b(\alpha)}. \]
  The Thurston metric is complete, proper, geodesic, induces the usual topology on $\T$, and satisfies the Busemann axiom; see \cite{Thu86}. This allows us to identify $\T$ with its image in $C(\T,\RR)$, under the map $z \mapsto \Phi_z$.
  
  Walsh showed that the horofunction compactification of $(\T,\dL)$ and the Thurston compactification of \T coincide. 

  \begin{theorem}[\cite{Wal14}] \label{thm:horo}
     
     A sequence $z_n \in \T$ converges in the Thurston compactification of
     \T if and only if $\Phi_{z_n}$ converges in the horofunction
     compactification of \T. If the limit in the Thurston compactification
     is the projective class $[\eta] \in \PML$, then the limiting
     horofunction is \[ \Phi_\eta(x) =  \log \sup_{\alpha \in \curve}
     \frac{\I(\eta,\alpha)}{\ell_x (\alpha)} -  \log \sup_{\alpha \in \curve}
     \frac{ \I(\eta,\alpha) }{\ell_b(\alpha)} .\]

  \end{theorem}

   We use the following normalization. Identify $\PML$ with
    \[
    P_b=\{\lambda\in\ML:\ell_b(\lambda)=1\}.
    \]

  For $z \in \T$ and $\eta\in\ML^*$, set
    \[
    Q(z)=\max_{\lambda\in P_b}\ell_z(\lambda)
    =e^{\dL(b,z)},
    \qquad
    Q(\eta)=\max_{\lambda\in P_b}i(\eta,\lambda),
    \]
    and define
    \[
    L_z(\lambda)=\frac{\ell_z(\lambda)}{Q(z)},
    \qquad
    L_\eta(\lambda)=\frac{i(\eta,\lambda)}{Q(\eta)}.
    \]
    Using this notation, we can succinctly write, 
    \begin{equation} \label{eqn:normalized}
    \Phi_z(x)
    =
    \log\max_{\lambda\in\ML^*}
    \frac{L_z(\lambda)}{\ell_x(\lambda)},
    \qquad
    \Phi_\eta(x)
    =
    \log\max_{\lambda\in\ML^*}
    \frac{L_\eta(\lambda)}{\ell_x(\lambda)}.
    \end{equation}
    
  We will use the following convergence criterion of Walsh:
    
  \begin{proposition}[\cite{Wal14}] \label{prop:convergence}
    Let $z_n \in \T$ and $[\eta] \in \PML$. Then $z_n \to [\eta]$ if and only if $L_{z_n} \to L_\eta$ uniformly on compact subsets of $\ML$. 
  \end{proposition}

    \begin{lemma}\label{lem:divergence}
    If $x_n\to[\xi]\in\PML$, then $Q(x_n)\to\infty$.
    \end{lemma}
    
    \begin{proof}
    This follows from properness of the metric. Let $r_n>0$ satisfy
    $r_n\ell_{x_n}(\cdot)\to i(\xi,\cdot)$. Then $r_n\to0$.
    Choose $\alpha\in P_b$ with $i(\xi,\alpha)>0$. It follows that
    $\ell_{x_n}(\alpha)\to\infty$, and hence
    \[
    Q(x_n)\ge \ell_{x_n}(\alpha)\to\infty. \qedhere
    \]
    \end{proof}
  
 \subsection{Proof of Theorem \ref{introthm:stars}}

  For $[\xi]\in\PML$, recall
  \[ Z \big( [\xi] \big) =\set{[\eta] \in \PML}{\I(\eta,\xi) = 0}.\] The condition $i(\param,\param)=0$ is well-defined on \PML, and continuity of the intersection pairing on $\ML$ implies that $Z \big([\xi]\big)$ is a closed subset of \PML.

  The proof of Theorem \ref{introthm:stars} is broken into two directions. Corollary \ref{cor:direction1} shows $\star\big([\xi]\big) \subset Z\big( [\xi] \big)$, and Corollary \ref{cor:direction2} shows $Z\big( [\xi] \big) \subset \star_b\big([\xi]\big)$. Together these yield \[Z\big( [\xi] \big) \subset \star_b\big([\xi]\big) \subset \star \big( [\xi] \big) \subset Z \big( [\xi] \big),\] establishing equality throughout. 

  Henceforth, to simplify notation, we will write $A \prec B$ if $A \le KB$
  for some suitable constant $K$. The notation $A \asymp B$ will mean $A \prec B$ and $B \prec A$. 

 \subsection{Short markings and the first direction}
  
  We first recall two useful estimates for short markings. A \emph{marking} $\mu$ on $S$ consists of a pants decomposition together with a transversal $\bar\alpha$ for each pants curve $\alpha$. We also write $\alpha = \bar{\bar{\alpha}}$. A \emph{short marking} $\mu_x$ on $x \in \T$ is obtained by successively choosing shortest pants curves and then shortest transversals. Note that a short marking on $x$ may not be unique, but all short markings on $x$ have uniformly bounded intersection number with each other. 

  There is $K_S>0$ such that, for every short marking $\mu_x$ and every $\alpha\in\mu_x$,
  \begin{equation} \label{eqn:curve_and_dual}
    \ell_x(\alpha)\ell_x(\bar\alpha)\le K_S.
  \end{equation}
   Moreover, for every $\eta\in\ML$,
   \begin{equation} \label{eqn:LRT}
    \ell_x(\eta)
    \asymp
    \sum_{\alpha\in\mu_x}
    \ell_x(\bar\alpha)i(\eta,\alpha),
    \end{equation}
    with multiplicative constants depending only on $S$. \cite{LRT12} proved Equation (\ref{eqn:LRT}) for curves. By density of weighted curves in $\ML$ and continuity of the intersection pairing and length function, the estimate extends to all of $\ML$.

  \begin{lemma} \label{lem:nondisjointness}
    Let $x_n\to[\xi]$, let $\mu_n$ be a short marking on $x_n$, and set
    \[ \lambda_{n,\alpha} =
    \frac{\ell_{x_n}(\bar\alpha)}{Q(x_n)} \alpha, \quad 
    \alpha\in\mu_n.
    \]
    If $\eta\in\ML$ satisfies $i(\xi,\eta)>0$, then:
   \begin{itemize}
     \item[(i)] The set $\bigcup_n \left\{ \lambda_{n,\alpha} \right\}_{\alpha \in \mu_n}$ lies in a fixed compact subset of $\ML$.
     \item[(ii)] There exists $\alpha_n \in \mu_n$ such that:
     \begin{itemize}
     \item[(a)] $\inf_n i (\eta,\lambda_{n,\alpha_n}) > 0$; and 
     \item[(b)] for every $z_n\to[\eta]$, there exists $N$ such that $\inf_{n \ge N} L_{z_n}(\lambda_{n,\alpha_n}) >0.$
     \end{itemize}
    \end{itemize}
  \end{lemma}

  \begin{proof}
    By (\ref{eqn:LRT}), for some $A=A(S)$ and every $\lambda\in\ML$,
    \begin{align*} \label{eqn:LRT2}
         A^{-1}\, \ell_{x_n}(\lambda) \le \sum_{\alpha \in \mu_n} \ell_{x_n}
     (\bar{\alpha}) \I(\lambda, \alpha) \le A\, \ell_{x_n}(\lambda). 
    \end{align*}
    Dividing the above by $Q(x_n)$ we obtain:
    \begin{equation} \label{eqn:LRT3}
         A^{-1}\, L_{x_n}(\lambda) \le \sum_{\alpha \in \mu_n}  \I \big (\lambda, \lambda_{n,\alpha} \big) \le A\, L_{x_n}(\lambda). 
    \end{equation}

    Let $\gamma=\{\gamma_1,\ldots,\gamma_k\}$ be a finite filling
    collection of curves, such as a marking, and set
    \[ J(\lambda) := i(\gamma, \lambda) = \sum_i i(\gamma_i, \lambda), \quad \lambda \in \ML. \] 
    Since $\gamma$ is filling, sublevel sets of $J$ are compact. 

    By substituting $\lambda=\gamma_i$ in  (\ref{eqn:LRT3}), we have, for any $\beta \in \mu_n$
    \begin{align*}
      J(\lambda_{n,\beta}) 
      &\le \sum_{\alpha \in \mu_n} J \big( \lambda_{n,\alpha} \big) = \sum_{\alpha \in \mu_n} \sum_i i \big(\gamma_i, \lambda_{n,\alpha} \big) 
      =  \sum_i \sum_{\alpha \in \mu_n} i \big(\gamma_i, \lambda_{n,\alpha} \big) \le A
      \sum_i L_{x_n}(\gamma_i).
    \end{align*}
    By Proposition \ref{prop:convergence}, for each $i$, \[ L_{x_n}(\gamma_i) \to L_\xi(\gamma_i) = \frac{i(\xi,\gamma_i)}{Q(\xi)},\] so the right hand side of the above is bounded. This shows the existence of $B_0>0$ such that \[ J(\lambda_{n,\alpha}) \le B_0, \quad \forall \alpha \in \mu_n \text{ and } \forall n. \] Hence $\bigcup_n \{\lambda_{n,\alpha}\}_{\alpha \in \mu_n}$ lies in a compact subset of $\ML$.

    To prove (ii), apply (\ref{eqn:LRT3}) with $\lambda=\eta$. Since \[ L_{x_n}(\eta) \to L_{\xi}(\eta)= \frac{i(\eta,\xi)}{Q(\xi)} > 0, \] and $\# \mu_n=6g-6+2p$ is fixed,  there exist $B>0$ and $\alpha_n \in \mu_n$ such that \[ i \big(\eta, \lambda_{n,\alpha_n} \big) \ge B > 0, \,\, \forall n. \] 
    
    Let $\mathcal{K} \subset \ML$ be a compact set containing all $\lambda_{n,\alpha_n}$. If $z_n \to [\eta]$, then  \[ \sup_{\lambda \in \mathcal{K}} \left| L_{z_n}(\lambda) - L_\eta(\lambda) \right| \to 0. \] Hence, for all sufficiently large $n$, we have \[ L_{z_n}(\lambda_{n,\alpha_n}) \ge \frac{1}{2} L_\eta(\lambda_{n,\alpha_n}) \ge \frac{1}{2} \frac{i(\eta,\lambda_{n,\alpha_n})}{Q(\eta)} \ge \frac{B}{2 Q(\eta)} > 0. \qedhere\] 
  \end{proof}

   \begin{proposition} \label{prop:divergence}
     Suppose $[\xi],[\eta]\in\PML$ satisfy $i(\xi,\eta)>0$. For any sequences $x_n\to[\xi]$ and $z_n\to[\eta]$, $\Phi_{z_n}(x_n)\to\infty.$
   \end{proposition}

  \begin{proof}
    Let $\mu_n$ be a short marking on $x_n$, and let $ \lambda_n=\lambda_{n,\alpha_n}$ be given by the preceding lemma. Then $L_{z_n}(\lambda_n)\ge B>0$ eventually,
    while (\ref{eqn:curve_and_dual})  gives
    \[
    \ell_{x_n}(\lambda_n)
    =
    \frac{\ell_{x_n}(\bar\alpha_n)\ell_{x_n}(\alpha_n)}
    {Q(x_n)}
    \le
    \frac{K_S}{Q(x_n)}.
    \]
    Combining the two bounds above, we obtain for all $n \ge N$:
    \begin{align*}
        e^{\Phi_{z_n}(x_n)} = \max_{\lambda \in \ML^*} \frac{L_{z_n}(\lambda)}{\ell_{x_n}(\lambda)} \ge \frac{L_{z_n}(\lambda_n}{\ell_{x_n} \big( \lambda_n)} \ge \frac{B}{K_S} Q(x_n).
    \end{align*} 
    By Lemma \ref{lem:divergence}, $Q(x_n) \to \infty$, so $\Phi_{z_n}(x_n) \to \infty$ as desired. \qedhere
  \end{proof}

  We immediately obtain:
  
  \begin{corollary} \label{cor:direction1}
      $\star^{\text{Th}}\big( [\xi] \big) \subset Z \big( [\xi] \big) $.
  \end{corollary}

  \begin{proof}
    We prove the contrapositive. If $i(\xi,\eta)>0$, Proposition \ref{prop:divergence} shows that for every
    $x_n\to[\xi]$ and $z_n\to[\eta]$,
    \[
    \dL(x_n,z_n)-\dL(b,z_n)
    =
    \Phi_{z_n}(x_n)\to\infty.
    \]  
    Lemma \ref{lem:necc_star} therefore implies $[\eta] \notin \star^{\text{Th}} \big( [\xi],C \big)$ for any $C \ge 0$. Equivalently, $\bigcup_C \star^{\text{Th}}\big( [\xi],C \big) \subset Z\big( [\xi] \big)$. Since $Z \big( [\xi] \big)$ is closed, we obtain \[ \star^{\text{Th}} \big( [\xi] \big) = \overline{\bigcup_{C \ge 0} \star^{\text{Th}} \big( [\xi],C \big)}\subset \overline{Z\big([\xi]\big)} = Z\big([\xi]\big). \qedhere \]

  \end{proof}
 
\subsection{Shearing coordinates and the second direction}
   
   \begin{proposition} \label{prop:disjoint}

     Suppose $i(\xi,\eta)=0$. Then for every neighborhood $V \subset \TC$ of $[\xi]$,
    \[
    \inf_{x\in V\cap \T}\Phi_\eta(x)=-\infty.
    \]
   \end{proposition}

   \begin{proof}
    
     Fix $K\ge 1$. Because $i(\xi,\eta)=0$, the two measured laminations may be realized simultaneously, and every positive linear combination \[ F_t = e^t \xi + K \eta\] is again a measured lamination. In what follows, we view $F_t$ equivalently as a measured foliation.

     Choose $\mu \in \ML$ such that $i(\mu,\xi) > 0$. Then \[ i(\mu,F_t) = e^t i(\mu,\xi) + K \, i(\mu,\eta) \to \infty.\] Thus $F_t$ leaves every compact subset of $\MF(S)$. Moreover, we have $[F_t] \to [\xi]$ since
     \[ e^{-t} F_t = \xi + K e^{-t} \eta  \to \xi.\]

    If $S$ is closed, choose a maximal uniquely ergodic measured lamination $\lambda \in \ML$ with projective class distinct from $[\xi]$ and $[\eta]$. Then Corollary \ref{cor:unique_erg} gives $F_t \in \MF(\lambda)$. If $S$ is punctured, choose a maximal ideal geodesic lamination $\lambda$. Since $F_t$ has compact support, every parametrized leaf of $F_t$ intersects the ideal triangulation infinitely often. Meanwhile, every leaf of $\lambda$ tends to a puncture on both sides, so $F_t$ is totally transverse to $\lambda$. Proposition \ref{prop:transversality} gives $F_t \in \MF(\lambda)$.
     
     Let $x_t \in \T$ have $\lambda$--shearing coordinates $F_t$. By Theorem \ref{thm:papa}, $x_t \to [\xi]$, and for every curve $\alpha$, we have $i(F_t,\alpha) \le \ell_{x_t}(\alpha)$. Therefore
     \[
     \ell_{x_t}(\alpha) 
      \ge i(F_t,\alpha) = e^t i(\xi, \alpha) + K i(\eta,\alpha) \ge K i(\eta,\alpha).
      \]
     This gives  us 
      \[\sup_{\alpha \in \curve} \frac{i(\eta,\alpha)}{\ell_{x_t}(\alpha)} \le \frac{1}{K} \quad \Longrightarrow \quad \sup_{\alpha \in \curve} \frac{L_\eta(\alpha)}{\ell_{x_t}(\alpha)} \le \frac{1}{K Q(\eta)}. \]
    
    Since $K\ge 1$ was arbitrary, for any $M>0$ we can find a path $x_t \to [\xi]$ such that \[ \Phi_\eta(x_t) = \log \sup_{\alpha \in \curve} \frac{L_\eta(\alpha)}{\ell_{x_t}(\alpha)} \le \log \frac{1}{KQ(\eta)} \le -M.\] Since $x_t \in V$ for all sufficiently large $t$, the result follows. \qedhere
      
   \end{proof}
   
   The second direction is now easily obtained.
   
   \begin{corollary}  \label{cor:direction2}
    
     $Z \big( [\xi] \big) \subset \star_b^{\text{Th}} \big( [\xi] \big)$. 

   \end{corollary}

   \begin{proof}
     Let $i(\eta,\xi)=0$. By Proposition \ref{prop:disjoint}, for all neighborhoods $V$ of $[\xi]$, \[ \inf_{x \in V \cap \T} \Phi_\eta(x) = -\infty < 0.\] Therefore, by Lemma \ref{lem:suff_star}, we have $[\eta] \in \star^{\text{Th}} \big( [\xi],0\big) = \star_b^{\text{Th}} \big([\xi]\big).$
   \end{proof}

  \subsection{Applications} \label{sec:proofs}
  
  Theorem \ref{introthm:stars} immediately gives Corollary \ref{introcor:hyperbolic_points}. 

  \begin{proof}[Proof of Corollary \ref{introcor:PA}]
  
    Recall the complexity of a surface $S=S_{g,p}$ is $3g-3+p$. If $f \in \mcg$ is pseudo-Anosov, its limit set consists of its stable and unstable projective laminations, which are uniquely ergodic and hence hyperbolic. If $f$ is instead infinite-order reducible with reducing multicurve $C$, then every boundary accumulation point $[\mu] \in \PML$ of an $\langle f \rangle$-orbit satisfies $i(\mu,C)=0.$ Consequently, $Z \big([\mu]\big)$ contains the projective class of every component of $C$. If $[\mu]$ is not the projective class of a component of $C$, then $Z \big( [\mu] \big)$ is not a singleton. If $[\mu]$ is the class of a component of $C$, then $Z \big( [\mu] \big)$ is again not a singleton when $S$ has complexity at least two. When $S$ has complexity one, some power of $f$ is a power of a Dehn twist, and hence $f$ has only one boundary accumulation point. Thus, in every case, $f$ is not strictly hyperbolic. Finally, a finite-order mapping class has no boundary accumulation points. \qedhere
    
  \end{proof}
  
  \begin{proof}[Proof of Corollary \ref{introcor:isometries}]
     Let $f: \T \to \T$  be an isometry of the Thurston metric. By Proposition \ref{prop:assumptions} and Theorem \ref{thm:horo}, the map $f$ induces a homeomorphism
     $ \partial f: \PML \to \PML$ of the horofunction boundary.

    Since isometries preserve stars and Theorem \ref{introthm:stars} identifies stars with zero sets, for every $ [\xi]\in\PML$, we have
    \[ \partial f\Bigl(Z \big( [\xi] \big)\Bigr) = Z \Bigl(\partial f \big( [\xi] \big)\Bigr).\]
    The zero set $Z \big( [\xi] \big)$ has codimension one in $\PML$ if and only if $[\xi]$ is represented by a simple closed curve; see for instance \cite{OP18,Alb19}. Since homeomorphisms preserve topological dimension, $\partial f$ preserves the subset
    of curves $\curve \subset \PML$. 
    Moreover, $\partial f$ preserves the star relation, which agrees with the disjointness relation on $\curve$. Hence $\partial f$ induces an automorphism of the curve graph $\cc$. By \cite{Iva97,Luo00}, there exists an extended mapping class $g$ whose action on $\cc$ agrees with the action induced by $\partial f$. It follows that the boundary map induced by $g^{-1} f$ acts trivially on $\curve$. Since $\curve \subset \PML$ is dense, this boundary map is the identity on \PML. By the argument of \cite[Lemma~7]{Wal14}, $g^{-1} f$ is the identity on $\T$. \qedhere
   \end{proof}

\section{\Teich metric} \label{sec:Teich}

  We now equip $\T$ with the \Teich metric $\dT$, while retaining the Thurston compactification. The \Teich distance is defined by
  \[ \dT(x,y) = \frac{1}{2} \log \inf_\varphi \text{K}(\varphi), \]
  where $\varphi:x\to y$ ranges over $K$--quasiconformal homeomorphisms in the appropriate homotopy class. This metric is symmetric, complete, proper, and geodesic. Kerckhoff's formula states that \[ \dT(x,y) = \frac12\log \sup_{\alpha\in \curve} \frac{\Ext_y(\alpha)}{\Ext_x(\alpha)}, \] where extremal length extends continuously to $\ML$ and is homogeneous of degree two \cite{Ker80}. Consequently Kerckhoff's formula can equivalently be maximized over any compact section of $\PML$ in $\ML$.

  We will use the standard estimates \cite{Ker80,Mas85}:
  \begin{equation} \label{eqn:ext-hyp}
      \Ext_x(\mu) \ge \frac{\ell_x(\mu)^2}{2\pi|\chi(S)|}, \qquad \mu\in\ML,
  \end{equation}
  and, for every curve $\alpha$, 
  \begin{equation} \label{eqn:maskit}
      \Ext_x(\alpha) \le \frac{\ell_x(\alpha)}{2} e^{\ell_x(\alpha)/2}.
  \end{equation}
   

\begin{lemma}\label{lem:ext-hyp}
Let $x_n\in\T$ be a sequence converging to
$[\xi]\in\PML$, and let
$\lambda_n\in\ML$ be a sequence converging to
$\lambda\in\ML$. Then, $i(\xi,\lambda)>0$ implies 
\[
\Ext_{x_n}(\lambda_n)\longrightarrow\infty
\]
In particular, if the sequence $\Ext_{x_n}(\lambda_n)$ is bounded, then $i(\xi,\lambda)=0$.
\end{lemma}

 \begin{proof}

    Set \[ s_n=\frac{Q(\xi)}{Q(x_n)}. \] 
    
    By Proposition \ref{prop:convergence} and Lemma \ref{lem:divergence}, 
    \[ s_n\to0 \qquad\text{and}\qquad s_n\ell_{x_n}(\lambda_n)\to i(\xi,\lambda), \] where we use uniform convergence on the compact set $\left \{\lambda\}\cup\{\lambda_n:n\ge1 \right\}$. 
    
    Consequently, if $i(\xi,\lambda)>0$, then $\ell_{x_n}(\lambda_n)\to\infty$, and by applying (\ref{eqn:ext-hyp}) $\Ext_{x_n}(\lambda_n)\to\infty$.  \qedhere 
    \end{proof}
    
  We also use the following standard estimates for Dehn twists. 

  \begin{lemma} \label{lem:twist}
      Let $\alpha$ be a curve and $T_\alpha$ the Dehn twist about $\alpha$.
    \begin{enumerate} 
      \item[(i)] For every $L>0$ there is $C=C(L)$ such that if $\ell_x(\alpha)\le L$, then for all $k \in \ZZ$, $T_\alpha^k$ has a quasiconformal representative on $x$ with dilatation \[ K\le C(1+|k|)^2. \] 
      \item[(ii)] There are $C=C(S)$ and $\epsilon_0>0$ such that if $\ell_x(\alpha)\le\epsilon_0$, then for all $k \in \ZZ$, $T_\alpha^k$ has a quasiconformal representative with dilatation \[ K\le C\bigl(1+|k|\ell_x(\alpha)\bigr)^2. \] 
    \end{enumerate} 
  \end{lemma}

  \begin{proof}
    Part (i) follows from the standard affine twist construction \cite[Section 2.4]{Lei04}, and part (ii) from Minsky's product region theorem \cite[Theorem 6.1]{Min96}. \qedhere
  \end{proof}

  \begin{lemma}\label{lem:twist-length}
    Let $\gamma$ be a simple closed curve and let $y_n\in\T$ be a
    sequence which degenerates only along $\gamma$: namely, fix a pants decomposition P containing $\gamma$, and suppose that all Fenchel–Nielsen coordinates of $y_n$
	with respect to P are fixed except 
    \[
    \ell_{y_n}(\gamma)\to 0.
    \]
    
    Let $A$ be a multicurve with support $
    \alpha_1,\ldots,\alpha_p$
    such that $i(A,\gamma)=0$, and let
    \[
    T_n=\prod_{i=1}^p T_{\alpha_i}^{m_{i,n}},
    \quad m_{i,n}\geq 0.
    \]
    We allow $\gamma$ to be one of the components of $A$. Let $l_n$ denote the radius of the standard collar about $\gamma$ on $y_n$. 
    
    Then for every curve $\delta$ there is $C_\delta>0$, independent of $n$ and of the $m_{i,n}$, such that
    \[
    \left|
    \ell_{y_n}(T_n\delta)
    -
    \sum_{i=1}^p
    m_{i,n}\ell_{y_n}(\alpha_i)i(\alpha_i,\delta)
    -
    2l_n i(\gamma,\delta)
    \right|
    \leq C_\delta.
    \]
    \end{lemma}

    \begin{proof}
Consider the multicurve
\[
\widehat A=\{\gamma\}\cup\operatorname{supp}(A).
\]
For each $\beta\in\widehat A$, let $U_{\beta,n}$ be its standard
collar in $y_n$, of radius $l_{\beta,n}$. Let $m_{\beta,n}$ denote
the exponent of $T_\beta$ in the multitwist $T_n$, where a missing
factor is understood to have exponent zero. Equivalently,
\[
m_{\beta,n}:=
\begin{cases}
m_{i,n}, & \text{if }\beta=\alpha_i\text{ for some }
i\in\{1,\ldots,p\},\\
0, & \text{if }\beta=\gamma\notin\operatorname{supp}(A).
\end{cases}
\]


     Since $y_n$ degenerates only along $\gamma$, for every $\beta \neq \gamma$ the geometry of $U_{\beta,n}$ stays in a compact family; in particular $l_{\beta,n}=O(1)$. The boundary lengths of all the collars are uniformly bounded, including $U_{\gamma,n}$ by Lemma \ref{lem:collar}.

     Cut the geodesic representative of $\delta$ along these collars. The portions outside the collars have total length $O_\delta(1)$. If $\sigma$ is a crossing arc of $U_{\beta,n}$, then
     \[ |w(\sigma,\beta)|\,\ell_{y_n}(\beta)=O_\delta(1). \]
      Here $w(\sigma,\beta)$ denotes the winding number of $\sigma$ about $\beta$, as defined immediately before Lemma \ref{lem:winding}. For $\beta\neq\gamma$, this follows from the bounded geometry of the $\beta$-collar, while for $\beta=\gamma$ it follows from the fact that the Fenchel--Nielsen twist coordinate about $\gamma$ is fixed.
      
      After applying $m_{\beta,n}$ twists, the winding number becomes $w(\sigma,\beta)\pm m_{\beta,n}$, and hence
      \[ \Bigl||w(\sigma,\beta)\pm m_{\beta,n}|-m_{\beta,n}\Bigr|
         \,\ell_{y_n}(\beta) \leq |w(\sigma,\beta)|\ell_{y_n}(\beta) = O_\delta(1).
      \]
      Lemma \ref{lem:winding} therefore gives, uniformly in $n$ and $m_{\beta,n}$,
      \[ \ell_{y_n}(\sigma) = 2l_{\beta,n} + m_{\beta,n}\ell_{y_n}(\beta) + O_\delta(1) \]
      for the geodesic representative of the twisted relative homotopy class.
        
     Summing over the $i(\beta,\delta)$ crossings and the complementary
        arcs gives
        \[
        \ell_{y_n}(T_n\delta)
        \le
        \sum_{i=1}^p
        m_{i,n}\ell_{y_n}(\alpha_i)i(\alpha_i,\delta)
        +
        2l_n i(\gamma,\delta)
        +
        O_\delta(1),
        \]
        where the terms $2l_{\beta,n}i(\beta,\delta)$ for $\beta\neq\gamma$
        are absorbed into the error.
        
        For the reverse inequality, cut the geodesic representative of
        $T_n\delta$ along the same collar boundaries. The endpoints of the complementary pieces may differ from those of the corresponding pieces of $\delta$. Sliding them along the collar boundaries to the latter endpoints has total cost $O_\delta(1)$, since the boundary lengths are uniformly bounded and the number of endpoints depends only on $\delta$. After this adjustment, the complementary pieces may be replaced by those of $\delta$, and the resulting collar arcs have winding numbers $w(\sigma,\beta)\pm m_{\beta,n}$. By the same application of Lemma~\ref{lem:winding} we get
        \[
        \ell_{y_n}(T_n\delta)
        \ge
        \sum_{i=1}^p
        m_{i,n}\ell_{y_n}(\alpha_i)i(\alpha_i,\delta)
        +
        2l_n i(\gamma,\delta)
        -
        O_\delta(1).
        \]
        Combining the two inequalities proves the lemma. \qedhere
    
    \end{proof}
  
  \subsection{Proof of Theorem \ref{introthm:false}} 
  
  Recall that \[ Z^{(2)}([\xi]) = \left\{ [\eta]\in\PML: \exists[\lambda]\in\PML \text{ with } i(\xi,\lambda)=i(\eta,\lambda)=0 \right\}. \] This set is closed, being the projection of a closed subset of the compact space $\PML\times\PML$. 

  The proof of Theorem \ref{introthm:false} is also divided into two directions. Proposition \ref{prop:direction1} shows $Z^{(2)}\big( [\xi] \big) \subset \star_b^T \big([\xi]\big)$, while Proposition \ref{prop:direction2} shows $\star^T \big([\xi]\big) \subset Z^{(2)} \big( [\xi] \big)$.  Together these yield \[Z^{(2)} \big( [\xi] \big) \subset \star_b^T \big([\xi]\big) \subset \star^T\big( [\xi] \big) \subset Z^{(2)} \big( [\xi] \big),\] establishing equality throughout. 
  
  The main tool is the following lemma on multicurves. 
  
  \begin{lemma} \label{lem:false}
    
    Suppose $A$ and $B$ are weighted multicurves and $\gamma$ is a curve such that \[ i(A,\gamma)=i(B,\gamma)=0. \] Then there exist $x_n, z_n \in \T$ such that $x_n \to [A]$ and $z_n \to [B]$ and \[ \dT(x_n,z_n) - \dT(b,z_n) \to - \infty.\]
    
  \end{lemma} 

  \begin{proof}
    Choose a pants decomposition containing $\gamma$, and let
    $y_n\in \T$ have fixed Fenchel--Nielsen coordinates except
    \[
    \ell_{y_n}(\gamma)=e^{-n^2}.
    \]
    Set $k_n=e^n$, and write
    \[
    A=\sum_{i=1}^p a_i\alpha_i,
    \qquad
    B=\sum_{j=1}^q b_j\beta_j.
    \]
    Define
    \[
    m_{i,n}
    =
    \left\lfloor
    \frac{a_i k_n}{\ell_{y_n}(\alpha_i)}
    \right\rfloor,
    \qquad
    r_{j,n}
    =
    \left\lfloor
    \frac{b_j k_n}{\ell_{y_n}(\beta_j)}
    \right\rfloor,
    \]
Note that 
\begin{equation}\label{eq:floorfunction}
0 \leq a_ik_n - m_{i,n} \ell_{y_n}(\alpha_i) < \ell_{y_n}(\alpha_i)
\qquad
\text{ and }
\qquad
0 \leq b_jk_n - r_{j,n} \ell_{y_n}(\beta_j) < \ell_{y_n}(\beta_j)
\end{equation}
We consider the following sequence in $\T$
    \[
    x_n=\prod_i T_{\alpha_i}^{-m_{i,n}}y_n,
    \qquad
    z_n=\prod_j T_{\beta_j}^{-r_{j,n}}y_n.
    \]
    
    We first claim that $x_n\to[A]$ and $z_n\to[B]$. Since every
    component of $A$ and $B$ is disjoint from $\gamma$, its $y_n$-length
    is uniformly bounded unless it is $\gamma$ itself. Hence by Eq~\ref{eq:floorfunction}
    \[
    m_{i,n}\ell_{y_n}(\alpha_i)=a_i k_n+O(1),
    \qquad
    r_{j,n}\ell_{y_n}(\beta_j)=b_j k_n+O(1).
    \]
    The collar radius of $\gamma$ is $O(n^2)$, so Lemma \ref{lem:twist-length} gives, for
    every curve $\delta$,
    \[
    \ell_{x_n}(\delta)
    =
    k_n i(A,\delta)+O_\delta(n^2+1),
    \qquad
    \ell_{z_n}(\delta)
    =
    k_n i(B,\delta)+O_\delta(n^2+1).
    \]
    Since $n^2/k_n\to0$, it follows that
    \[
    x_n\to[A],
    \qquad
    z_n\to[B].
    \]
    
    We next estimate their \Teich distance. For every twisting
    curve $\rho\neq\gamma$, we have
    $\ell_{y_n}(\rho)$ is uniformly bounded above and away from $0$, while for $\rho=\gamma$ the corresponding twist exponent $s_n$ satisfies
    \[
    s_n\ell_{y_n}(\gamma)=O(k_n).
    \]
    Thus Lemma \ref{lem:twist} and the triangle inequality give
    \[
    \dT(x_n,z_n) \le \dT(x_n,y_n)+\dT(y_n,z_n)
    =
    O(\log k_n)=O(n).\]
    
    On the other hand, all twists used to construct $z_n$ fix $\gamma$,
    so
    \[
    \ell_{z_n}(\gamma)=e^{-n^2}.
    \]
    By (\ref{eqn:maskit}) and Kerckhoff's formula,
    \[
    \dT(b,z_n)
    \ge
    \frac12
    \log
    \frac{\Ext_b(\gamma)}{\Ext_{z_n}(\gamma)}
    \geq
    \frac{n^2}{2}-O(1).
    \]
    Consequently,
    \[
    \dT(x_n,z_n)-\dT(b,z_n)
    \le
    O(n)-\frac{n^2}{2}+O(1)
    \longrightarrow -\infty.
    \]
    \end{proof}

  \begin{proposition} \label{prop:direction1}
      $Z^{(2)} \big( [\xi] \big) \subset \star_b^T \big( [\xi] \big)$. 
  \end{proposition}

  \begin{proof}
    
    Let $[\eta]\in Z^{(2)}([\xi])$. We will construct weighted
    multicurves $A_n,B_n$ and curves $\gamma_n$ such that
    \[
    [A_n]\to[\xi],\qquad [B_n]\to[\eta],
    \]
    and
    \[
    i(A_n,\gamma_n)=i(B_n,\gamma_n)=0.
    \]

    First suppose $i(\xi,\eta)=0$. If $\xi$ is non-filling, then
    Lemma \ref{lem:non-filling}, applied with $\lambda=\xi$, gives a curve $\gamma$ such
    that $i(\xi,\gamma)=0$ and
    \[
    i(\mu,\gamma)>0\Longrightarrow i(\mu,\xi)>0, \quad \forall \mu\in\ML.
    \]
    Since $i(\eta,\xi)=0$, we also have $i(\eta,\gamma)=0$. 
    The same conclusion follows by interchanging $\xi$ and $\eta$ if
    $\eta$ is non-filling. Weighted multicurves with zero intersection
    with $\gamma$ are dense in $Z([\gamma])$, so we may choose
    $A_n,B_n$ such that
    \[
    [A_n]\to[\xi],\qquad [B_n]\to[\eta],
    \qquad
    i(A_n,\gamma)=i(B_n,\gamma)=0.
    \]
    In this case set $\gamma_n=\gamma$.

    It remains to consider the case where $\xi$ and $\eta$ are filling.
    Since $i(\xi,\eta)=0$, they have the same support.
    By Theorem \ref{thm:simultaneous}, there are weighted multicurves $A_n$ and $B_n$
    with the same support such that
    \[
    [A_n]\to[\xi],\qquad [B_n]\to[\eta].
    \]
    Let $\gamma_n$ be any component of their common support. Then
    \[
    i(A_n,\gamma_n)=i(B_n,\gamma_n)=0.
    \]

    Now suppose $i(\xi,\eta)>0$. By the definition of
    $Z^{(2)}([\xi])$, there is $\lambda\in\ML^*$ such that
    \[
    i(\xi,\lambda)=i(\eta,\lambda)=0.
    \]
    The lamination $\lambda$ is non-filling, for otherwise
    $i(\xi,\eta)=0$. By Lemma \ref{lem:non-filling} there is a curve $\gamma$ such that
    \[
    i(\xi,\gamma)=i(\eta,\gamma)=0.
    \]
    As above, choose weighted multicurves $A_n,B_n$ with
    \[
    [A_n]\to[\xi],\qquad [B_n]\to[\eta],
    \qquad
    i(A_n,\gamma)=i(B_n,\gamma)=0,
    \]
    and set $\gamma_n=\gamma$.
    
    Thus in every case we have weighted multicurves $A_n,B_n$ and
    curves $\gamma_n$ satisfying
    \[
    [A_n]\to[\xi],\qquad [B_n]\to[\eta],
    \qquad
    i(A_n,\gamma_n)=i(B_n,\gamma_n)=0.
    \]
    For each $n$, Lemma \ref{lem:false} gives sequences
    $x_{n,k},z_{n,k}\in \T$ such that
    \[
    x_{n,k}\to[A_n],\qquad z_{n,k}\to[B_n],
    \]
    and
    \[
    d_T(x_{n,k},z_{n,k})-d_T(b,z_{n,k})\to-\infty
    \]
    as $k\to\infty$. Lemma \ref{lem:semicontinuity} therefore gives
    \[
    [\eta]\in\star_b^T \big( [\xi] \big). \qedhere
    \]
    \end{proof}

  \begin{proposition} \label{prop:direction2}
      $\star^T \big( [\xi] \big) \subset Z^{(2)} \big( [\xi] \big) $. 
  \end{proposition}
  
  \begin{proof}
  
    We prove the contrapositive. Suppose $[\xi],[\eta] \in \PML$ are jointly filling: equivalently, \[ i(\xi,\lambda)+i(\eta,\lambda)>0,  \quad \forall \lambda\in\ML^*.\] Let $x_n, z_n \in \T$ be any sequences with $x_n \to [\xi]$ and $z_n \to [\eta]$. Our goal is to show \[ \limsup_{n \to \infty} \dT(x_n,z_n) - \dT(b,z_n) = \infty.\] 

    Set \[ E_b = \{\lambda\in\ML:\Ext_b(\lambda)=1\}, \] a compact section of $\PML$. By symmetry and Kerckhoff's formula, \[ \dT(b,z_n) = \dT(z_n,b) = \frac12\log \max_{\lambda\in E_b} \frac{1}{\Ext_{z_n}(\lambda)}. \] 
    Choose $\lambda_n\in E_b$ realizing this maximum. Then \[ \Ext_{z_n}(\lambda_n) = e^{-2\dT(b,z_n)}. \] Since $z_n$ converges to a point of the Thurston boundary, it leaves every compact subset of $\T$, and hence $\dT(b,z_n)\to\infty$. Therefore $\Ext_{z_n}(\lambda_n)\to 0.$ 
    
    After passing to a subsequence, say $n_i$, compactness of $E_b$ gives \[ \lambda_{n_i}\to\lambda\in E_b. \] 

    Applying Lemma \ref{lem:ext-hyp} to $z_{n_i}\to[\eta]$ and $\lambda_{n_i}\to\lambda$, we obtain $i(\eta,\lambda)=0.$ Since $\xi$ and $\eta$ are jointly filling, it follows that  $i(\xi,\lambda)>0.$ Applying Lemma \ref{lem:ext-hyp} again, now to $x_{n_i}\to[\xi]$, gives $\Ext_{x_{n_i}}(\lambda_{n_i})\to \infty.$ 
    Hence 
    \[ 
    \begin{aligned} 
      \dT(x_{n_i},z_{n_i})-\dT(b,z_{n_i}) &\ge \frac12 \log \frac{\Ext_{x_{n_i}}(\lambda_{n_i})} {\Ext_{z_{n_i}}(\lambda_{n_i})} - \frac12 \log \frac{1}{\Ext_{z_{n_i}}(\lambda_{n_i})} \\ &= \frac12\log\Ext_{x_{n_i}}(\lambda_{n_i}) \longrightarrow\infty. 
    \end{aligned} 
    \] 
    
    Thus \[ \limsup_{n\to\infty} \bigl(\dT(x_n,z_n)-\dT(b,z_n)\bigr) = \infty. \] 

     Since the sequences $x_n\to[\xi]$ and $z_n\to[\eta]$ were arbitrary, Lemma \ref{lem:necc_star} implies that  \[ \star^T([\xi],C)\subset Z^{(2)}([\xi]), \quad \text{for every }C\ge0. \] Since $Z^{(2)}([\xi])$ is closed, we obtain \[ \star^T \big( [\xi] \big) = \overline{\bigcup_C \star^T \big( [\xi], C \big)}  \subset \overline{Z^{(2)} \big( [\xi] \big)} = Z^{(2)} \big( [\xi] \big). \qedhere\]
    
\end{proof}

 
  \bibliographystyle{alpha}
  \bibliography{main}
  \end{document}